\documentclass[a4paper,10pt]{article}
\usepackage{amssymb,amsmath,amsthm}
\usepackage{fullpage}
\usepackage{hyperref}
\usepackage{eucal}
\usepackage{color}
\usepackage{stackrel}
\usepackage{graphicx}
\newcommand{\ie}{\emph{i.e.}}
\newcommand{\eg}{\emph{e.g.}}
\newcommand{\cf}{\emph{cf.}}
\newcommand{\Real}{\mathbb{R}}

\newcommand{\Nat}{\mathbb{N}}
\newcommand{\Int}{\mathbb{Z}}
\newcommand{\Sphere}{\mathbb{S}}

\newcommand{\Dom}{\mathsf{D}}
\newcommand{\dist}{\mathop{\mathrm{dist}}\nolimits}
\newcommand{\rot}{\mathop{\mathrm{rot}}\nolimits}
\newcommand{\vertiii}[1]{{\left\vert\kern-0.25ex\left\vert\kern-0.25ex\left\vert #1 
    \right\vert\kern-0.25ex\right\vert\kern-0.25ex\right\vert}}

\newcommand{\sii}{L^2}

\newcommand{\der}{\mathrm{d}}
\def\OMIT#1{}
\newtheorem{Theorem}{Theorem}

\newtheorem{Proposition}{Proposition}

\theoremstyle{definition}
\newtheorem{Remark}{Remark}

\usepackage[normalem]{ulem}
\definecolor{DarkGreen}{rgb}{0,0.5,0.1} % David

\newcommand\soutD{\bgroup\markoverwith
{\textcolor{DarkGreen}{\rule[.5ex]{2pt}{1pt}}}\ULon}
\newcommand\soutP{\bgroup\markoverwith
{\textcolor{blue}{\rule[.5ex]{2pt}{1pt}}}\ULon}
\newcommand{\Hm}[1]{\leavevmode{\marginpar{\tiny%
$\hbox to 0mm{\hspace*{-0.5mm}$\leftarrow$\hss}%
\vcenter{\vrule depth 0.1mm height 0.1mm width \the\marginparwidth}%
\hbox to
0mm{\hss$\rightarrow$\hspace*{-0.5mm}}$\\\relax\raggedright #1}}}

\begin{document}
%
%-------%
% TITLE %
%-------%
%------------------------------------------%
%------------------------------------------%
\title{\textbf{\LARGE
On the improvement of the Hardy inequality 
due to singular magnetic fields
}}
\author{Luca Fanelli,$^a$ \
David Krej\v{c}i\v{r}{\'\i}k,$^b$ \
Ari Laptev$^c$ \ and \ 
Luis Vega$^d$}
\date{\small 
\begin{quote}
\emph{
\begin{itemize}
\item[$a)$] 
Dipartimento di Matematica, SAPIENZA Universit\`a di Roma,
P.~le Aldo Moro 5, 00185 Roma;
fanelli@mat.uniroma1.it.%
\item[$b)$] 
Department of Mathematics, Faculty of Nuclear Sciences and 
Physical Engineering, Czech Technical University in Prague, 
Trojanova 13, 12000 Prague 2, Czechia;
david.krejcirik@fjfi.cvut.cz.%
\\
\item[$c)$]
Department of Mathematics, Imperial College London, 
Huxley Building, 180 Queen’s Gate, London SW7 2AZ, United Kingdom;
a.laptev@imperial.ac.uk%
\\
\item[$d)$]
Departamento de Matem\'aticas, Universidad del Pais Vasco, 
Aptdo.~644, 48080 Bilbao, \&
Basque Center for Applied Mathematics (BCAM), 
Alameda Mazarredo 14, 48009 Bilbao, Spain; 
luis.vega@ehu.es \& lvega@bcamath.org.
\end{itemize}
}
\end{quote}
11 July 2018}
\maketitle
\begin{abstract}
\noindent
We establish magnetic improvements upon the classical Hardy inequality
for two specific choices of singular magnetic fields.
First, we consider the Aharonov-Bohm field in all dimensions
and establish a sharp Hardy-type inequality that takes into
account both the dimensional as well as the magnetic flux contributions.
Second, in the three-dimensional Euclidean space,
we derive a non-trivial magnetic Hardy inequality 
for a magnetic field that vanishes at infinity and diverges along a plane.
\end{abstract}
%
%------------------------------------------%
%------------------------------------------%
%
%\vspace{-4ex}
%
%\noindent

%---------------------%
\section{Introduction}
%---------------------%
%
The subcriticality of the Laplacian in $\Real^d$ for $d \geq 3$
can be quantified by means of the classical Hardy inequality 
\begin{equation}\label{Hardy}
  -\Delta \geq \left(\frac{d-2}{2}\right)^2 \frac{1}{r^2}
\end{equation}
valid in the sense of quadratic forms in $\sii(\Real^d)$,
where~$-\Delta$ is the standard self-adjoint realisation of 
the Laplacian in $\sii(\Real^d)$
and~$r$ is the distance to the origin of~$\Real^d$.
On the other hand, the Laplacian is critical in~$\Real$ and~$\Real^2$
in the sense that the spectrum of the shifted operator $-\Delta + V$ 
starts below zero whenever the operator of multiplication~$V$
is bounded, non-positive and non-trivial. 
In quantum mechanics, 
interpreting~$-\Delta$ as the Hamiltonian of a free electron,
the Hardy inequality~\eqref{Hardy} 
can be interpreted as the uncertainty principle
with important consequences for
the stability of atoms and molecules.

Inequality~\eqref{Hardy} goes back to 1920 \cite{Hardy_1920}
and it is well known that it is optimal in the sense
that the dimensional constant is the best possible 
and no other non-negative term could be added 
on the right-hand side of~\eqref{Hardy}. 
A much more recent observation is that adding any magnetic field 
leads to an improved Hardy inequality, including dimension $d=2$.
A variant of this statement is the magnetic Hardy inequality
\begin{equation}\label{Hardy.magnet}
  (-i\nabla+A)^2 
  -\left(\frac{d-2}{2}\right)^2 \frac{1}{r^2}
  \geq \frac{c_{d,B}}{1+r^2\log^2(r)}
  \,,
\end{equation}
valid in the sense of quadratic forms in $\sii(\Real^d)$ for $d \geq 2$.
Here $A:\Real^d\to\Real^d$ is a smooth vector potential
and~$c_{d,B}$ is a non-negative constant that depends only
on the dimension~$d$ and the magnetic field $B=\der A$;
the constant~$c_{d,B}$ is positive if, and only if, 
the field~$B$ is not identically equal to zero.
This inequality was first proved by Laptev and Weidl in 1999
\cite{Laptev-Weidl_1999} for $d=2$ under an extra flux condition,
in which case the lower bound holds 
with a better weight (without the logarithm)
on the right-hand side of~\eqref{Hardy.magnet}.
A general version of~\eqref{Hardy.magnet} is due 
to Cazacu and Krej\v{c}i\v{r}{\'\i}k~\cite{CK},
but we also refer to~\cite{Weidl_1999},
\cite{Alziary-Fleckinger-Pelle-Takac_2003},
\cite{Balinsky-Laptev-Sobolev_2004}, 
\cite[Sec.~6]{Kovarik_2011}
and~\cite{Ekholm-Portmann_2014} for previous related works.

The principal motivation of the present paper is
our curiosity about the structure of the magnetic improvement 
on the right-hand side of~\eqref{Hardy.magnet}.
Our first result deals with the Aharonov-Bohm potential
\begin{equation}\label{AB}
  A_\alpha(x,y,z_1,\dots,z_{d-2}) 
  := \alpha \left(\frac{-y}{x^2+y^2},\frac{x}{x^2+y^2},0,\dots,0\right)
  , \qquad
  \alpha \in\Real \,,
\end{equation}
where $(x,y,z_1,\dots,z_{d-2}) \in \Real^d$.
We abbreviate $z:=(z_1,\dots,z_{d-2}) \in \Real^{d-2}$
and denote by $\rho(x,y,z) := \sqrt{x^2+y^2}$
the distance of a point $(x,y,z) \in \Real^d$
to the subspace $\{x=y=0\} \subset \Real^d$ of dimension~$d-2$. 
Let us also recall that $r(x,y,z) := \sqrt{x^2+y^2+|z|^2}$
denotes the distance of $(x,y,z) \in \Real^d$
to the origin of~$\Real^d$.
Because of the singularity of~$A_\alpha$ at the origin,
it is important to specify the self-adjoint realisation
of the associated magnetic Laplacian;
we customarily understand $(-i\nabla+A_\alpha)^2$ as the Friedrichs 
extension of this operator initially defined 
on $C_0^\infty(\Real^d\setminus\{\rho=0\})$.
\begin{Theorem}\label{Thm.Ari}
Let~$A_\alpha$ be given by~\eqref{AB}.
For every $\alpha \in \Real$, one has
\begin{equation}\label{main.Ari} 
  (-i\nabla+A_\alpha)^2 
  - \left(\frac{d-2}{2}\right)^2 \frac{1}{r^2}
  \geq \frac{\dist(\alpha,\Int)^2}{\rho^2}
\end{equation}
in the sense of quadratic forms in $\sii(\Real^d)$ with $d \geq 2$.
\end{Theorem}

This theorem in dimension $d=2$ is due to 
Laptev and Weidl \cite{Laptev-Weidl_1999}.
The novelty of Theorem~\ref{Thm.Ari} consists
in the present extension to the higher dimensions, $d \geq 3$.
The result is optimal in the sense that
that the constants appearing in~\eqref{main.Ari}
are the best possible 
and no other non-negative term could be added 
on the right-hand side of the inequality. 
In other words, subtracting the right-hand side 
from the left-hand side, the obtained operator would be critical.
Notice also that the flux-type condition $\alpha \not\in \Int$ 
is necessary to have the subcriticality
of the operator on the left-hand side of~\eqref{main.Ari}.
Indeed, $(-i\nabla+A_\alpha)^2$ is unitarily equivalent
to the magnetic-free Laplacian whenever $\alpha \in \Int$, 
so in this case the criticality of the operator 
on the left-hand side of~\eqref{main.Ari} follows
from the optimality of the classical Hardy inequality~\eqref{Hardy}.

Because of the special form of the vector potential~\eqref{AB},
the operator $(-i\nabla+A_\alpha)^2$ 
admits a natural decomposition with respect to the variables
$(x,y) \in \Real^2$ and $z \in \Real^{d-2}$.	
However, it is important to stress that~\eqref{main.Ari}
does not follow as a result of this separation of variables.
In fact, while the right-hand side of~\eqref{main.Ari}
is the two-dimensional contribution coming from the angular
component of the magnetic Laplacian in the $(x,y)$-plane,
the second (dimensional) term on the left-hand side 
of the inequality is a contribution coming from 
both the radial component of 
the magnetic Laplacian in the $(x,y)$-plane
as well as the Laplacian in the $z$-space. 

The feature of the Aharonov-Bohm potential~\eqref{AB}
is that its singularity is supported 
on a manifold of codimension two.
Our next interest lies in a vector potential
with a singularity supported on a hyperplane. 
In this case, we restrict our attention 
to the the three-dimensional toy model 
\begin{equation}\label{model}
  A_\beta(x,y,z) := \beta \left(\frac{y}{z^2},0,0\right)
  , \qquad
  \beta \in\Real \,,
\end{equation}
where $(x,y,z) \in \Real^3$.
The seemingly simple choice~\eqref{model} for the vector potential
is of course very special, but on the other hand 
the model is intrinsically three-dimensional 
in the sense that no reduction to lower dimensions 
via a separation of variables is available.

We understand the magnetic Laplacian
$-\Delta_\beta := (-i\nabla+A_\beta)^2$
corresponding to~\eqref{model} with $\beta \in\Real\setminus\{0\}$
as the Friedrichs extension of this operator initially defined 
on $C_0^\infty(\Real^d\setminus\{z=0\})$.
Because of the strong singularity of~$A_\beta$ on the plane $\{z=0\}$,
the unperturbed version of~$-\Delta_\beta$  
is not the Laplacian~$-\Delta$ in~$\Real^3$
but rather the Dirichlet Laplacian in $\Real^3 \setminus \{z=0\}$
that we denote by~$-\Delta_0$.
More specifically, the singularity of the potential requires 
that the functions from the operator domain 
of~$-\Delta_\beta$ vanish on $\{z=0\}$,
representing thus certain confinement of the electron
to one of the two half-spaces $\{z>0\}$ or $\{z<0\}$.
The unperturbed operator~$-\Delta_0$ satisfies the Hardy inequality
\begin{equation}\label{Hardy.zero} 
  -\Delta_0 \geq \frac{1}{4}\frac{1}{z^2}
  \,,
\end{equation}
which is optimal 
(in the same way as~\eqref{Hardy} is optimal for~$-\Delta$).
Notice that the distance to the origin in~\eqref{Hardy}
is replaced by the distance to the plane $\{z=0\}$ in~\eqref{Hardy.zero}.
Our next result shows that there is always a specific improvement
whenever $\beta \not= 0$. 

\begin{Theorem}\label{Thm.main}
Let~$A_\beta$ be given by~\eqref{model}.
For every $\beta \in \Real$,
one has
\begin{equation}\label{main} 
  -\Delta_\beta  
  - \frac{1}{4}\frac{1}{z^2}
  \geq |\beta| \left(
  1 + \frac{y^2}{z^2}
  \right) 
  \frac{1}{z^2} 
\end{equation}
in the sense of quadratic forms in $\sii(\Real^3)$.
\end{Theorem}

The first term on the right-hand side of~\eqref{main} is easy to obtain, 
while the second term is much less obvious.
Our strategy to identify the second improvement
is based on a test-function argument,
which we believe to be of independent interest.	
Contrary to Theorem~\ref{Thm.Ari},
we do not know whether the inequality of Theorem~\ref{Thm.main}
is optimal if $\beta \not= 0$. 

The rest of the paper naturally splits 
into two independent sections.
In Section~\ref{Sec.AB} we quickly prove Theorem~\ref{Thm.Ari},
while Theorem~\ref{Thm.main} is established 
in a longer Section~\ref{Sec.toy}.

%----------------------%
\section{The Aharonov-Bohm field}\label{Sec.AB}
%----------------------%
%
In this section we exclusively consider
the vector potential~$A_\alpha$ from~\eqref{AB}
with any $d \geq 2$. 

\subsection{Preliminaries}
For every $\alpha \in \Real$,
we introduce the magnetic Laplacian~$(-i\nabla+A_\alpha)^2$
as the self-adjoint non-negative operator in $\sii(\Real^d)$
associated with the quadratic form
\begin{equation*}
  Q_\alpha[\psi] := \|(\nabla+iA_\alpha)\psi\|^2 
  \,, \qquad
  \Dom(Q_\alpha) := 
  \overline{C_0^\infty(\Real^d\setminus\{\rho=0\})}^{\vertiii{\cdot}}
  \,,
\end{equation*}
where~$\|\cdot\|$ denotes the usual norm of $\sii(\Real^d)$ and
\begin{equation*}
  \vertiii{\psi}
  := \sqrt{\displaystyle \|(\nabla+iA_\alpha)\psi\|^2 + \|\psi\|^2 } 
  \,.
\end{equation*}
By the diamagnetic inequality, we have the form domain inclusion
$
  \Dom(Q_\alpha) 
  \subset W_0^{1,2}(\Real^d \setminus \{\rho=0\})
  = W^{1,2}(\Real^d)
$,
where the equality follows from the fact that 
the subset $\{\rho=0\} \subset \Real^d$ is a polar set
(\cf~\cite[Sec.~VIII.6]{Edmunds-Evans}).
Using the special structure~\eqref{model} of the potential~$A_\beta$, 
we have
\begin{equation*}
  Q_\alpha[\psi] =
  \int_{\Real^d}  
  \left(
  \left|\left(-i\partial_x 
  - \frac{\alpha \, y}{x^2+y^2} \right)\psi\right|^2
  + \left|\left(-i\partial_y 
  + \frac{\alpha \, x}{x^2+y^2} \right)\psi\right|^2
  + |\nabla_{\!z} \psi|^2  
  \right)
  \der x \, \der y \, \der z
  \,,
\end{equation*}
where $z = (z_1,\dots,z_{d-2})$
is a $(d-2)$-dimensional coordinate. 
In the sense of distributions, 
\begin{equation*} 
  (-i\nabla+A_\alpha)^2  = 
  \Big(-i\partial_x - \frac{\alpha \, y}{x^2+y^2} \Big)^2 
  + \Big(-i\partial_y + \frac{\alpha \, x}{x^2+y^2} \Big)^2
  -\Delta_z  
  \,,
\end{equation*}
where $-\Delta_z$ is the usual (distributional) Laplacian 
in the~$z$ variables.

If $\alpha=0$, then 
$
  \Dom(Q_0) = W^{1,2}(\Real^d)
$
and the operator associated with~$Q_0$
is just the standard self-adjoint realisation
of the Laplacian~$-\Delta$ in~$\sii(\Real^d)$.
More generally, if $\alpha \in \Int$ then $(-i\nabla+A_\alpha)^2$
is unitarily equivalent to the (magnetic-free) Laplacian $-\Delta$.
This can be seen as follows. 
Passing to the polar coordinates in the $(x,y)$-plane,
\ie\ writing $(x,y) = (\rho\cos\varphi,\rho\sin\varphi)$
with $\rho \in (0,\infty)$ and $\varphi \in (0,2\pi]$,
we have the obvious unitary equivalences
\begin{equation}\label{partial.wave} 
\begin{aligned}
  (-i\nabla+A_\alpha)^2  
  &\cong 
  - \rho^{-1} \, \partial_\rho \, \rho \, \partial_\rho
  + \frac{(-i\partial_\varphi + \alpha)^2}{\rho^2}  
  -\Delta_z  
  \\
  &\cong
  \bigoplus_{m\in\Int}
  \left(
  - \rho^{-1} \, \partial_\rho \, \rho \, \partial_\rho
  + \frac{\nu_m^2}{\rho^2} 
  \right)
  -\Delta_z  
  \,.
\end{aligned}
\end{equation}
Here $\nu_m := m+\alpha$ are the eigenvalues 
of the one-dimensional operator $-i\partial_\varphi + \alpha$
in $\sii([0,2\pi))$, subject to periodic boundary conditions.
The corresponding set of eigenfunctions read 
$\{e^{im\varphi}\}_{m\in\Int}$.
If~$\alpha$ is an integer, then the direct sum is indistinguishable  
from the usual partial-wave decomposition of the Laplacian~$-\Delta$.

Finally, let us notice that the spectrum of $(-i\nabla+A_\alpha)^2$
equals the semiaxis $[0,\infty)$ for every real~$\alpha$.

\subsection{The improved Hardy inequality}
Let $\psi \in C_0^\infty(\Real^d\setminus\{\rho=0\})$,
a core of~$Q_\alpha$.
Employing the polar coordinates in the $(x,y)$-plane
as in~\eqref{partial.wave} and writing 
$
  \phi(\rho,\varphi,z) =: \psi(\rho\cos\varphi,\rho\sin\varphi,z)
$,
we have 
\begin{equation}\label{angular.estimate} 
\begin{aligned}
  Q_\alpha[\psi]
  &= \int_{\Real^{d-2}} \int_0^{2\pi} \int_0^\infty 
  \left(
  |\partial_\rho\phi|^2 
  + \frac{|\partial_\varphi\phi+i\alpha\phi|^2}{\rho^2}
  + |\nabla_{\!z}\phi|^2
  \right)
  \rho \, \der\rho \, \der\varphi \, \der z
  \\
  &\geq 
  \int_{\Real^{d-2}} \int_0^{2\pi} \int_0^\infty 
  \left(
  |\partial_\rho\phi|^2 
  + \frac{\dist(\alpha,\Int)^2}{\rho^2} |\phi|^2
  + |\nabla_{\!z}\phi|^2
  \right)
  \rho \, \der\rho \, \der\varphi \, \der z
  \,,
\end{aligned}
\end{equation}
where we omit to specify the arguments of~$\phi$
and abuse a bit the notation for~$\rho$. 
This inequality explains the quantity 
on the right-hand side of~\eqref{main.Ari}.
To obtain the dimensional term on the left-hand side of~\eqref{main.Ari},
we write
$$
  \phi(\rho,\varphi,z) = f(\rho,\varphi,z) \, 
  (\rho^2 + |z|^2)^{-(d-2)/4}
  \,,
$$
which is in fact the definition of the new test function~$f$.
Notice that $f(0,\varphi,z)=0$ for all $\varphi \in [0,2\pi)$
and $z \in \Real^{d-2}$.
A straightforward computation employing 
an integration by parts yields
\begin{equation}\label{radial.estimate} 
\begin{aligned}
  \int_{\Real^{d-2}}  \int_0^\infty 
  \left(
  |\partial_\rho\phi|^2 
  + |\nabla_{\!z}\phi|^2
  \right)
  \rho \, \der\rho \, \der z
  = \ & \left(\frac{d-2}{2}\right)^2  
  \int_{\Real^{d-2}}  \int_0^\infty 
  \frac{|\phi|^2}{\rho^2+|z|^2} \, \rho \, \der\rho \, \der z
  \\
  & \ + \int_{\Real^{d-2}}  \int_0^\infty 
  \left(
  |\partial_\rho f|^2 
  + |\nabla_{\!z} f|^2
  \right)
  \, (\rho^2 + |z|^2)^{-(d-2)/2} \, \rho \, \der\rho \, \der z
  \\
  \geq \ & 
  \left(\frac{d-2}{2}\right)^2  
  \int_{\Real^{d-2}}  \int_0^\infty 
  \frac{|\phi|^2}{\rho^2+|z|^2} \, \rho \, \der\rho \, \der z
  \,.
\end{aligned}
\end{equation}
Estimates~\eqref{angular.estimate} and~\eqref{radial.estimate} 
yield~\eqref{main.Ari}, after coming back to the Cartesian
coordinates and noticing that $\rho^2+|z|^2 = r^2$.
This concludes the proof of Theorem~\ref{Thm.Ari}.

The present proof also explains 
why the inequality~\eqref{main.Ari} is optimal.
Indeed, the inequality~\eqref{angular.estimate} is sharp
in the sense that it is achieved by any function 
of the form $\phi(\rho,\varphi,z) = g(\rho,z) e^{im\varphi}$,
where $m\in\Int$ is chosen in such a way 
that it minimises the distance $\dist(\alpha,\Int)$
(so that $e^{im\varphi}$ is an eigenfunction corresponding 
to the lowest eigenvalue $\dist(\alpha,\Int)^2$ 
of the operator $(-i\partial_\varphi + \alpha)^2$).  
The other inequality~\eqref{angular.estimate} is not achieved 
by a non-trivial~$\phi$, but it is also sharp in the following sense.
For any function~$f$ depending only on~$r$,
we have
$$
\begin{aligned}
  \int_{\Real^{d-2}}  \int_0^\infty 
  \left(
  |\partial_\rho f|^2 
  + |\nabla_{\!z} f|^2
  \right)
  \, (\rho^2 + |z|^2)^{-(d-2)/2} \, \rho \, \der\rho \, \der z
  &= \int_{\Sphere_+^{d-2}}  \int_0^\infty 
  |\partial_r f|^2 
  \, \rho \, \der r \, \der \sigma
  \\
  &\leq 
  |\Sphere_+^{d-2}|  
  \int_0^\infty  |\partial_r f|^2  \, r \, \der r 
  \,,
\end{aligned}
$$
where $\Sphere_+^{d-2} := \Sphere^{d-2} \cap \{\rho>0\}$
and $\Sphere^{d-2}$ is the unit sphere in the $(\rho,z)$-half-space.  
It is well known that there exists a sequence of functions 
$\{f_n\}_{n=1}^\infty \subset C_0^\infty((0,\infty))$
such that 
$$
  f_n(r) \xrightarrow[n\to\infty]{} 1
  \mbox{ pointwise}
  \qquad \mbox{and} \qquad
  \int_0^\infty  |\partial_r f_n(r)|^2  \, r \, \der r 
  \xrightarrow[n\to\infty]{} 0
  \,,
$$
for the integral corresponds to the radial part
of the two-dimensional Laplacian.

%----------------------------%
\section{The confining field}\label{Sec.toy}
%----------------------------%
%
The organisation of this section dealing with 
the vector potential~\eqref{model} is as follows.
In Section~\ref{Sec.pre} we rigorously introduce
the corresponding magnetic Laplacian~$-\Delta_\beta$ 
as a self-adjoint operator in~$\sii(\Real^3)$
and state its basic spectral properties.
The elementary part of Theorem~\ref{Thm.main}
(\ie~just the first term on the right-hand side of~\eqref{main})
is established in Section~\ref{Sec.elementary}.
The complete proof of Theorem~\ref{Thm.main}
is given in the remaining 
Sections~\ref{Sec.proof1} and~\ref{Sec.proof2}.

\subsection{Preliminaries}\label{Sec.pre}
For every $\beta \in \Real$,
we introduce the magnetic Laplacian~$-\Delta_\beta$ 
the self-adjoint non-negative operator in $\sii(\Real^3)$
associated with the quadratic form
\begin{equation*}
  Q_\beta[\psi] := \|(\nabla+iA_\beta)\psi\|^2 
  \,, \qquad
  \Dom(Q_\beta) := \overline{C_0^\infty(\Real^3\setminus\{z=0\})}^{\vertiii{\cdot}}
  \,,
\end{equation*}
where~$\|\cdot\|$ denotes the usual norm of $\sii(\Real^3)$ and
\begin{equation*}
  \vertiii{\psi}
  := \sqrt{\displaystyle \|(\nabla+iA_\beta)\psi\|^2 + \|\psi\|^2 } 
  \,.
\end{equation*}
By the diamagnetic inequality, we have the form domain inclusion
$
  \Dom(Q_\beta) \subset W_0^{1,2}(\Real^3 \setminus \{z=0\})  
$.
Using the special structure~\eqref{model} of the potential~$A_\beta$, 
we have
\begin{equation*}
  Q_\beta[\psi] =
  \int_{\Real^3}
  \left(
  \left|\left(-i\partial_x +\frac{\beta}{z^2} y\right)\psi\right|^2
  + |\partial_y \psi|^2 + |\partial_z \psi|^2
  \right)
  \der x \, \der y \, \der z
\end{equation*}
and, in the sense of distributions, 
\begin{equation}\label{sense}
  -\Delta_\beta = \Big(-i\partial_x+\frac{\beta}{z^2} y\Big)^2 
  -\partial_y^2 -\partial_z^2
  = -\Delta -	 2 i \frac{\beta}{z^2} y \, \partial_x + \frac{\beta^2}{z^4} y^2
  \,,
\end{equation}
where $-\Delta$ is the usual (distributional) Laplacian 
in the $(x,y,z)$ variables.

Notice that $-\Delta_0$ (\ie~$\beta=0$)
is just the Dirichlet Laplacian in $\Real^3\setminus\{z=0\}$
for which the form domain equality 
$
  \Dom(Q_0) = W_0^{1,2}(\Real^3 \setminus \{z=0\})  
$
holds.
The spectrum of~$-\Delta_0$ is well known,
$\sigma(-\Delta_0) = [0,\infty)$.
%
%\begin{equation*}
%  \sigma(-\Delta_0) = [0,\infty) \,. 
%\end{equation*}
%
Moreover, using the classical one-dimensional Hardy inequality
\begin{equation}\label{Hardy.1D}
  \forall u \in W_0^{1,2}(\Real\setminus\{0\})
  \,, \qquad
  \int_\Real |u'(z)|^2 \, \der z
  \geq \frac{1}{4} 
  \int_\Real \frac{|u(z)|^2}{z^2} \, \der z
  \,,
\end{equation}
and Fubini's theorem, it follows that~$-\Delta_0$ is subcritical
in the sense that the three-dimensional Hardy inequality 
$$
  \forall \psi \in W_0^{1,2}(\Real^3\setminus\{z=0\})
  \,, \qquad
  \int_{\Real^3} |\nabla\psi(x,y,z)|^2 \, \der x \, \der y \, \der z
  \geq \frac{1}{4} 
  \int_{\Real^3} \frac{|\psi(x,y,z)|^2}{z^2} \, \der x \, \der y \, \der z
  \,,
$$
holds. This is the precise statement of~\eqref{Hardy.zero}. 

It is not difficult to see that the spectrum of~$-\Delta_\beta$
coincides with the spectrum of the unperturbed operator~$-\Delta_0$
(as well as of the usual Laplacian without the extra Dirichlet condition).

\begin{Proposition}
For every $\beta \in \Real$, one has
$$
  \sigma(-\Delta_\beta) = [0,\infty) 
  \,.
$$
\end{Proposition}  
\begin{proof}
The inclusion $\sigma(-\Delta_\beta) \subset [0,\infty)$
follows trivially because of the non-negativity of $-\Delta_\beta$.
The opposite inclusion $\sigma(-\Delta_\beta) \supset [0,\infty)$
can be established by the Weyl criterion,
by choosing the singular sequence localised at the infinity
of the cone $\{|y| < |z|\}$, 
where the terms in~\eqref{sense} containing~$\beta$ 
can be made arbitrarily small.
We omit the details.
\end{proof}

\subsection{The elementary Hardy inequality}\label{Sec.elementary}
Now we turn to the elementary part of the Hardy inequality~\eqref{main}.
\begin{Proposition}\label{Prop.elementary}
For every $\beta \in \Real$, one has
\begin{equation}\label{elementary}
  \forall \psi \in \Dom(Q_\beta)
  \,, \qquad
  Q_\beta[\psi]
  \geq \left(\frac{1}{4} + |\beta|\right) 
  \int_{\Real^3} \frac{|\psi(x,y,z)|^2}{z^2} \, \der x \, \der y \, \der z
  \,.
\end{equation}
\end{Proposition}  
\begin{proof}
Since the result will be re-proved in the following subsection,
here we provide just a sketchy proof.

Writing
\begin{equation}\label{Landau}
  -\Delta_\beta = 
  \underbrace{\left(-i\partial_x+\frac{\beta}{z^2} y\right)^2 
  - \partial_y^2}_{-\Delta_\beta'} 
  - \partial_z^2
  \,,
\end{equation}
we notice that~$-\Delta_\beta'$ is the magnetic Laplacian in $\sii(\Real^2)$
corresponding to the two-dimensional vector potential
\begin{equation*} 
  A_\beta'(x,y) := \frac{\beta}{z^2} \, (y,0) 
  \,,
\end{equation*}
which depends parametrically on~$z$ (and~$\beta$).	
The corresponding two-dimensional magnetic field is constant
\begin{equation}\label{constant} 
  B_\beta'(x,y) := \rot A_\beta'(x,y) = -\frac{\beta}{z^2}  
  \,.
\end{equation}
The operator~$-\Delta_\beta'$ is the celebrated Landau Hamiltonian.

The spectral problem for~$-\Delta_\beta'$ is explicitly solvable.
The easiest way how to see it is to 
perform a partial Fourier transform with respect to the $x$-variable,
which yields a unitary equivalence  
\begin{equation}\label{shifted} 
  -\Delta_\beta' \cong
  \left(\xi+\frac{\beta}{z^2} y\right)^2
  - \partial_y^2
  \,,
\end{equation}
where $\xi\in\Real$ is the dual variable to~$x$.
Noticing that the right-hand side of~\eqref{shifted} 
is the Hamiltonian of a shifted harmonic oscillator 
(the shift can be handled as yet another unitary transform),
we get the familiar formula 
(the natural numbers~$\Nat$ contain zero in our convention)
\begin{equation*} 
  \sigma(-\Delta_\beta') =
  \frac{2|\beta|}{z^2} \left(\Nat+\frac{1}{2}\right)
  \,, \qquad
  \beta \not= 0
  \,.
\end{equation*}
Each point in the spectrum is an eigenvalue of infinite multiplicity
(Landau levels).
(If $\beta=0$, then $\sigma(-\Delta_\beta') = [0,\infty)$.)
In particular,
\begin{equation*} 
  \inf\sigma(\Delta_\beta') = \frac{|\beta|}{z^2}  
\end{equation*}
(which is trivially valid also for $\beta=0$).

Using the last result in~\eqref{Landau}, we get
\begin{equation}\label{crude} 
  -\Delta_\beta \geq -\partial_z^2 +  \frac{|\beta|}{z^2} 
  \geq \frac{1}{4 z^2} + \frac{|\beta|}{z^2} 
  \,,
\end{equation}
which is the desired result~\eqref{elementary}.
The second estimate in~\eqref{crude} follows 
from the classical Hardy inequality~\eqref{Hardy.1D},
by noticing that the form core consists of functions that vanish
on the plane $\{z=0\}$.
\end{proof}
\begin{Remark}
Proposition~\ref{Prop.elementary} can be alternatively proved also 
by a standard commutator trick 
(see, \eg, \cite[Sec.~2.4]{Balinsky-Laptev-Sobolev_2004}).
Let us denote by $\Pi_j := -i\partial_j + (A_\beta)_j$ 
with $j \in \{1,2,3\} \cong \{x,y,z\}$
the $j^\mathrm{th}$ component of the magnetic gradient.
Then one has the identity 
\begin{equation}\label{commutator} 
  \forall C_0^\infty(\Real^3 \setminus \{z=0\})
  \,, \qquad
  \|\Pi_j\psi\|^2 + \|\Pi_k\psi\|^2 
  = \|(\Pi_j \pm i \Pi_k)\psi\|^2 
  \pm \langle\psi,(B_\beta)_{jk}\psi\rangle 
  \,,
\end{equation}
for any pair $j,k \in \{1,2,3\}$,
where $ (B_\beta)_{jk} := \partial_j (A_\beta)_k - \partial_k (A_\beta)_j$
are the coefficients of the magnetic tensor $B_\beta := \der A_\beta$
and $\langle\cdot,\cdot\rangle$ denotes the inner product of $\sii(\Real^3)$.
In our case~\eqref{model}, we have
$$
  B_\beta = \beta
  \begin{pmatrix}
    0 & -1/z^2 & 2y/z^3 \\
    1/z^2  & 0 & 0 \\
    -2y/z^3  & 0 & 0  
  \end{pmatrix}
  \,.
$$
Using the formula~\eqref{commutator} with $(j,k) := (1,2)$,
one therefore obtains, 
for every $\psi \in C_0^\infty(\Real^3 \setminus \{z=0\})$,
$$
  Q_\beta[\psi] = \|\Pi_1\psi\|^2 + \|\Pi_2\psi\|^2 + \|\Pi_3\psi\|^2 
  \geq \pm \beta \left\| \frac{\psi}{z} \right\|^2 + \|\partial_z\psi\|^2 
  \geq \left( \pm \beta + \frac{1}{4} \right) 
  \left\| \frac{\psi}{z} \right\|^2
  \,,
$$ 
where the last inequality is due 
to the classical Hardy inequality~\eqref{Hardy.1D}.
Since the obtained result holds with either the plus or minus sign,
we arrive at~\eqref{elementary} for every 
$\psi \in C_0^\infty(\Real^3 \setminus \{z=0\})$.
By density, the result extends to all $\psi \in \Dom(Q_\beta)$.
\end{Remark}

\subsection{Quantification of the elementary Hardy inequality}
\label{Sec.proof1}
Our next goal is to show that~\eqref{elementary} can still be improved. 
To do so, we have to employ the terms we neglected 
in the crude estimates~\eqref{crude}.

Let $\psi \in C_0^\infty(\Real^3\setminus\{z=0\})$,
a core of~$Q_\beta$.
The function is implicitly assumed to depend on
the space variables $(x,y,z) \in \Real^3$
and for brevity we omit to specify the arguments in the integrals below. 

First of all, let us perform the partial Fourier transform 
with respect to the $x$-variable as in~\eqref{shifted}:
\begin{equation}\label{form1}
  Q_\beta[\psi] =
  \int_{\Real^3}
  \left(
  \left|\left(\xi +\frac{\beta}{z^2} y\right)\hat\psi\right|^2
  + |\partial_y \hat\psi|^2 + |\partial_z \hat\psi|^2
  \right)
  \der \xi \, \der y \, \der z
  \,.
\end{equation}
Notice that the transformed function 
$\hat\psi = \hat\psi(\xi,y,z)$ 
still vanishes in a neighbourhood of $\{z=0\}$.
 
In the second step, we make the change of test function
\begin{equation}\label{change2} 
  \hat\psi(\xi,y,z) = \sqrt{|z|} \, \phi(\xi,y,z) 
\end{equation}
to single out the first term on the right-hand side of~\eqref{elementary}.
Putting~\eqref{change2} into~\eqref{form1} and integrating by parts
with respect to~$z$,
we arrive at 

\begin{equation}\label{form2}
  Q_\beta[\psi] 
  - \frac{1}{4} 
  \left\|\frac{\psi}{z}\right\|^2
  =
  \int_{\Real^3}
  \left(
  \left|\left(\xi +\frac{\beta}{z^2} y\right)\phi\right|^2
  + |\partial_y \phi|^2 + |\partial_z \phi|^2
  \right)
  |z| \, \der \xi \, \der y \, \der z
  \,.
\end{equation}

From now on, let us assume that~$\beta$ is non-zero. 
Then 
\begin{equation*}
  \eta(\xi,y,z) := \exp\left(-\frac{|\beta|}{2z^2}(y-y_0)^2\right) 
  \qquad \mbox{with} \qquad
  y_0 := -\frac{z^2 \xi}{\beta}
\end{equation*}
is an eigenfunction of the operator
on the right-hand side of~\eqref{shifted} corresponding to 
the lowest eigenvalue.
In this two-dimensional context, 
the variable~$z$ is understood as a parameter
and~$\xi$ gives rise to the degeneracies. 
In the third step, we make the change of test function
\begin{equation}\label{change3} 
  \phi(\xi,y,z) =  \eta(\xi,y,z) \, \varphi(\xi,y,z) 
\end{equation}
to single out the second term on the right-hand side of~\eqref{elementary}.
Putting~\eqref{change3} into~\eqref{form2} and integrating by parts
with respect to~$y$,
we arrive at 
\begin{equation}\label{form3}
  Q_\beta[\psi] 
  - \left(\frac{1}{4} + |\beta|\right) 
  \left\|\frac{\psi}{z}\right\|^2
  =
  \int_{\Real^3}
  \left(
  |\partial_y \varphi|^2 + |\partial_z \phi|^2
  \right)
  \eta^2(\xi,y,z) \, |z| \, \der \xi \, \der y \, \der z
  \,.
\end{equation}

Since the right-hand side of~\eqref{form3} is non-negative,
we have just re-proved~\eqref{elementary}.

\begin{Remark}
If the operator associated with the form 
on the right-hand side of~\eqref{form3} had compact resolvent
(which is not true), we would immediately arrive at 
a (local) improved Hardy inequality (\cf~\cite{K6}).
In our case, however, we have to proceed more carefully.
\end{Remark}

\subsection{A test-function argument}\label{Sec.proof2}
Having~\eqref{form3} at our disposal, the next step
is to use the contribution of the term containing~$|\partial_z\phi|^2$.
First of all, notice that
\begin{equation*} 
  \partial_z\phi = \partial_z\varphi 
  + \varphi \, \frac{|\beta|}{z^3} (y^2-y_0^2)
  \,.
\end{equation*}
The main trick is to pick an arbitrary function $f:\Real^3 \to \Real$
and write
\begin{equation*} 
\begin{aligned}
  |\partial_z\phi|^2 
  &= \left| 
  \partial_z\varphi - f \varphi
  + \varphi \left(
  \frac{|\beta|}{z^3} (y^2-y_0^2) + f 
  \right)
  \right|^2
  \\
  &= \left| 
  \partial_z\varphi - f \varphi
  \right|^2
  + |\varphi|^2 
  \left(
  \frac{|\beta|}{z^3} (y^2-y_0^2) + f 
  \right)^2
  +2 \Re \left[ 
  (\partial_z\bar\varphi - f \bar\varphi)
  \varphi \left(
  \frac{|\beta|}{z^3} (y^2-y_0^2) + f 
  \right)
  \right]
  \\
  &= \left| 
  \partial_z\varphi - f \varphi
  \right|^2
  + |\varphi|^2 
  \left(
  \frac{\beta^2}{z^6} (y^2-y_0^2)^2 - f^2 
  \right)
  + (\partial_z |\varphi|^2) 
  \left(
  \frac{|\beta|}{z^3} (y^2-y_0^2) + f
  \right)
  \,.
\end{aligned}
\end{equation*}
Putting this expression into~\eqref{form3} and integrating by parts
with respect to~$z$,
we arrive at 
\begin{equation}\label{form4}
  Q_\beta[\psi] 
  - \left(\frac{1}{4} + |\beta|\right) 
  \left\|\frac{\psi}{z}\right\|^2
  =
  \int_{\Real^3}
  \left(
  |\partial_y \varphi|^2 + |\partial_z \varphi - f\varphi|^2
  + V_f |\varphi|^2
  \right)
  \eta^2(\xi,y,z) \, |z| \, \der \xi \, \der y \, \der z
  \,,
\end{equation}
where
\begin{equation}\label{potential}
  V_f(\xi,y,z) :=
  -\partial_z f - \frac{1}{z} f - 2 \frac{|\beta|}{z^3} (y^2-y_0^2) f
  - \frac{\beta^2}{z^6} (y^2-y_0^2)^2
  + 2 \frac{|\beta|}{z^4} (y^2+y_0^2)
  \,.
\end{equation}
Now we may play with many possible choices of~$f$ 
to try to make the potential~$V$ positive
and employ the following consequence of~\eqref{form4}: 
\begin{equation*}
  Q_\beta[\psi] 
  - \left(\frac{1}{4} + |\beta|\right) 
  \left\|\frac{\psi}{z}\right\|^2
  \geq
  \int_{\Real^3}
  V_f \, |\hat\psi|^2 \, \der \xi \, \der y \, \der z
  \,.
\end{equation*}

In particular, the special choice
\begin{equation}\label{choice} 
  f(\xi,y,z) := \frac{1}{2} \frac{|\beta|}{z^3} (y_0^2-y^2)
\end{equation}
leads to a spectacularly simple expression
\begin{equation}\label{both} 
  V_f(\xi,y,z) =
  \frac{|\beta|}{z^4} (y^2+y_0^2)
  \,.
\end{equation}
Neglecting the second term on the right-hand side of~\eqref{both}
and performing the inverse
Fourier transform in the $\xi$-variable,
we therefore obtain the following Hardy inequality
\begin{equation}\label{improved} 
  Q_\beta[\psi]
  - \left(\frac{1}{4} + |\beta|\right) 
  \left\|\frac{\psi}{z}\right\|^2
  \geq 
  |\beta| \int_\Real^3 
  \frac{y^2}{z^4} \, |\psi|^2 
  \, \der x \, \der y \, \der z
\end{equation}
which coincides with~\eqref{main} 
after the extension of
$\psi \in C_0^\infty(\Real^3\setminus\{z=0\})$
to all $\Dom(Q_\beta)$.
This concludes the proof of Theorem~\ref{Thm.main}.
\hfill\qed	

\begin{Remark}
Keeping both terms on the right-hand side of~\eqref{both}
and performing the inverse
Fourier transform in the $\xi$-variable,
we have established an interesting operator inequality
\begin{equation*}
  -\Delta_\beta  
  \geq 
  \frac{-\partial_x^2}{|\beta|} + 
  \left(
  \frac{1}{4} + |\beta| + |\beta| \, \frac{y^2}{z^2}
  \right) 
  \frac{1}{z^2} 
  \,.
\end{equation*}
\end{Remark}
%

%---------------------------%
\subsection*{Acknowledgment}
%---------------------------%
%
The research of D.K.\ was partially supported 
by the GACR grant No.\ 18-08835S
and by FCT (Portugal) through project PTDC/MAT-CAL/\-4334/\-2014.
The research of L.V.\ was partially supported by
ERCEA Advanced Grant 669689-HADE, MTM2014-53145-P, and IT641-13. 

%\vfill
%\newpage
%--------------%
% BIBLIOGRAPHY %
%--------------%
%
%\addcontentsline{toc}{section}{References}
\bibliography{bib}
\bibliographystyle{amsplain}

\end{document}